\documentclass[a4paper,11pt,twoside]{amsart}
\usepackage[T1]{fontenc}
\usepackage[utf8]{inputenc}
\usepackage[british]{babel}
\usepackage{amsmath,amssymb,amsthm,amscd}
\usepackage{enumerate}
\usepackage{verbatim}
\usepackage{lmodern}
\usepackage{color}
\usepackage{fullpage}
\usepackage{tikz-cd}
\usepackage{soul}

\theoremstyle{definition}

\newtheorem{thm}{Theorem}[section]
\newtheorem{dfn}[thm]{Definition}

\newtheorem{prp}[thm]{Proposition}

\author{Antje Dabeler}
\thanks{AD, MG, and TdL~were supported by the DFG -- Project-ID 427320536 -- SFB 1442, and under Germany's Excellence Strategy EXC 2044 390685587, Mathematics Münster: Dynamics--Geometry--Structure.}
\address{University of M\"unster, Mathematical Institute, Einsteinstra\ss{}e 62, 48149 M\"unster, Germany}
\email{antje.dabeler@uni-muenster.de \newline \phantom{------------------------------} maria.gerasimova@uni-muenster.de \newline \phantom{------------------------------} tim.delaat@uni-muenster.de}

\author{Emilie Mai Elki\ae{}r}
\thanks{}
\address{Department of Mathematics, University of Oslo, Norway}
\email{elkiaer@math.uio.no}

\author{Maria Gerasimova}
\thanks{}
\address{}
\email{}

\author{Tim de Laat}
\thanks{}
\address{}
\email{}

\title{Unitary $L^{p+}$-representations of almost automorphism groups}

\begin{document}

\begin{abstract}
Let $G$ be a locally compact group with an open subgroup $H$ with the Kunze--Stein property, and let $\pi$ be a unitary representation of $H$. We show that the representation $\widetilde{\pi}$ of $G$ induced from $\pi$ is an $L^{p+}$-representation if and only if $\pi$ is an $L^{p+}$-representation. We deduce the following consequence for a large natural class of almost automorphism groups $G$ of trees: For every $p \in (2,\infty)$, the group $G$ has a unitary $L^{p+}$-representation that is not an $L^{q+}$-representation for any $q < p$. This in particular applies to the Neretin groups.
\end{abstract}

\maketitle

\section{Induction of unitary $L^{p+}$-representations}

Let $G$ be a locally compact group with a closed subgroup $H$. Given a unitary representation $\pi$ of $H$, it can be induced to $G$ in order to get a unitary representation of $G$. Technically, this procedure is much easier if $H$ is not only closed, but also open in $G$. In that case, it is well known that the restriction $\widetilde{\pi} \vert_H$ (restricted to $H$ as a function) contains the representation $\pi$.

We briefly recall the construction of induced representations, and we refer to \cite{MR3012851,MR4249450} for further details. Let $H$ be an open subgroup of $G$, and let $\{t_i\}_{i \in I}$ be a maximal set in $G$ consisting of left coset representatives. Let $(\pi,V)$ be a unitary representation of $H$. Let $\widetilde{V}=\{ f \colon G \to V \mid f(gh)=\pi(h^{-1})f(g) \; \forall g \in G \; \forall h \in H \textrm{ and } \sum_{i \in I} \|f(t_i)\|^2 < \infty \}$. With respect to the inner product $\langle f,f' \rangle = \sum_{i \in I} \langle f(t_i),f'(t_i) \rangle$, this is a Hilbert space. The induced \hyphenation{rep-re-sen-ta-tion}representation $(\widetilde{\pi},\widetilde{V})$ is defined by $(\widetilde{\pi}(g)f)(g')=f(g^{-1}g')$.

Given a unitary representation $(\pi,V)$ of $G$, a function of the form $\pi_{v,w} \colon g \mapsto \langle \pi(g)v
,w \rangle$, where $v,w \in V$, is called a matrix coefficient of $\pi$. This short note is concerned with the behaviour of integrability properties of such matrix coefficients under the procedure of induction of representations from open subgroups.
\begin{dfn}
Let $p \in [1, \infty)$. A unitary representation $(\pi,V)$ of a locally compact group $G$ is an $L^{p}$-representation if there is a dense linear subspace 
$V_0 \subset V$ such that for all $v,w \in V_0$, we have $\pi_{v,w} \in L^p(G)$. The representation $\pi$ is an $L^{p+}$-representation if it is an $L^{p+\varepsilon}$-representation for all $\varepsilon > 0$.
\end{dfn}
The property of being an $L^p$-representation (or $L^{p+}$-representation) is preserved under induction of representations. We briefly recall the proof, which is based on \cite[Theorem 2.4]{MR3705441}, in which Wiersma shows this fact in the setting of discrete groups. The same proof works in our setting and we claim no originality here.
\begin{prp} \label{prp:inducedlp}
Let $G$ be a locally compact group, $H$ an open subgroup of $G$, and $p \in [1,\infty)$. Let $(\pi,V)$ be a unitary representation of $H$ and $(\widetilde{\pi},\widetilde{V})$ the associated induced representation. If $\pi$ is an $L^p$-representation (resp.~$L^{p+}$-representation), then $\widetilde{\pi}$ is an $L^p$-representation (resp.~$L^{p+}$-representation).
\end{prp}
\begin{proof}
Let $V_0$ be a dense linear subspace of $V$ such that for all $v,w \in V_0$, the matrix coefficient $\pi_{v,w}$ lies in $L^p(H)$. Let $\{t_i\}_{i \in I}$ be as above, and let $\widetilde{V}_0$ be the subspace of $\widetilde{V}$ consisting of functions $f$ for which $f(t_i) \in V_0$ for all $i\in I$ and such that $f(t_i)$ is non-zero for only finitely many $i \in I$. Then $\widetilde{V}_0$ is a dense linear subspace of $\widetilde{V}$.

For each $i\in I$, denote by $\widetilde{V}_0^i$ the linear subspace of $\widetilde{V}$ consisting of all functions $f$ with support in $t_iH$ and such that $f(t_i)\in V_0$. Then $\widetilde{V}_0$ consists of finite linear combinations of functions from (different) spaces $\widetilde{V}_0^i$. For $i,j\in I$ and $f\in\widetilde{V}_0^i$ and $f'\in\widetilde{V}_0^j$, consider the matrix coefficient $\widetilde{\pi}_{f,f'}$. Set $w=f(t_i) \in V_0$ and $w'=f'(t_j) \in V_0$, and denote the modular function on $G$ (and its restriction to $H$) by $\Delta$. Then 
\begin{align*}
\|\widetilde{\pi}_{f,f'}\|^p_p &= \int_G \, \Big\lvert \langle \widetilde{\pi}(g)f,f' \rangle \Big\rvert^p \, d\mu_G(g) 
= \int_G \, \Big\lvert \sum_{k \in I} \langle f(g^{-1}t_k),f'(t_k) \rangle \Big\rvert^p \, d\mu_G(g) \\
&= \int_G \, \Big\lvert \langle f(g^{-1}t_j),w' \rangle \Big\rvert^p \, d\mu_G(g) 
= \Delta(t_i^{-1})\int_H \, \Delta(h^{-1}) \, \Big\lvert \langle f(t_ih),w' \rangle \Big\rvert^p \, d\mu_H(h) \\
&= \Delta(t_i^{-1})\int_H \, \Big\lvert \langle \pi(h)w,w' \rangle \Big\rvert^p \, d\mu_H(h) 
= \Delta(t_i^{-1}) \, \|\pi_{w,w'}\|^p_p < \infty.
\end{align*}
It follows that $\widetilde{\pi}$ is an $L^p$-representation.
\end{proof}

In general, it may happen that the induced representation has better integrability properties than the representation from which one induces, in the sense that the induced representation of an $L^p$-representation may be an $L^q$-representation for some $q < p$ (even if one considers the ``optimal'' value of $p$). Indeed, this for example happens if one induces $L^p$-representations of $\mathbb{F}_2$ for large $p$ to higher-rank simple Lie groups, since for every higher rank simple Lie group $G$, there exists $p(G) < \infty$ (depending on the group) such that every nontrivial unitary irreducible representation of $G$ is an $L^{p(G)}$-representation \cite[Th\'eor\`eme 2.4.2]{MR0560837}. This phenomenon is not well understood and we do not know whether it can actually happen in the case of induction from an open subgroup.

The main point of this note is to show that in a specific case, namely for representations of a locally compact group $G$ induced from representations of an open subgroup $H$ with the Kunze--Stein property, we know that the (optimal) $L^p$-integrability of a representation is preserved under induction. A locally compact group $G$ has the Kunze--Stein property if for every $p\in [1,2)$, the convolution product on $C_c(G)$ extends to a bounded bilinear map $L^p(G)\times L^2(G)\to L^2(G)$. This property originated from the work of Kunze and Stein \cite{MR0163988}, who proved it for $\mathrm{SL}(2,\mathbb{R})$.

\begin{thm}\label{thm:KSgp_inducedLpplus}
Let $G$ be a locally compact group, $H$ an open subgroup of $G$, and $p \in [1,\infty)$. Let $(\pi,V)$ be a unitary representation of $H$ and $(\widetilde{\pi},\widetilde{V})$ the associated induced representation of $G$. If $H$ has the Kunze-Stein property, then $\pi$ is an $L^{p+}$-representation of $H$ if and only if $\widetilde{\pi}$ is an $L^{p+}$-representation of $G$.
\end{thm}
\begin{proof}
One direction is Proposition \ref{prp:inducedlp}. For the other direction, suppose that $\widetilde{\pi}$ is an $L^{p+}$-representation of $G$. Then $\widetilde{\pi}|_H$ (restriction as a function) is an $L^{p+}$-representation of $H$. Because $\pi$ is contained in $\widetilde{\pi}|_H$ and $H$ has the Kunze-Stein property, it follows from Proposition \cite[Proposition 4.3]{MR4398257} that $\pi$ is itself an $L^{p+}$-representation.
\end{proof}
If we do not assume $H$ to have the Kunze--Stein property, the restriction $\widetilde{\pi}|_H$ might theoretically have better integrability properties than being an $L^{p+}$-representation. It would be interesting to have a better understanding of the relation between $L^{p+}$-representations of a group and $L^{p+}$-representations of its subgroups (also in the case of closed subgroups).

\section{Application to representations of almost automorphism groups of trees}
For $d \geq 3$, let $T_d$ be a $d$-regular tree and $\mathrm{Aut}(T_d)$ the automorphism group of $T_d$, which carries a natural non-discrete, totally disconnected, locally compact group topology. In \cite{MR1209033}, Neretin introduced the group $N_d$ of almost automorphisms of $T_d$. The Neretin groups $N_d$ also carry a natural non-discrete, totally disconnected, locally compact group topology, and $\mathrm{Aut}(T_d)$ embeds into $N_d$ as an open subgroup. It is known that the groups $N_d$ are simple \cite{MR1703086} and that they do not have any lattices \cite{MR2881324}. Recently, the Neretin groups have attracted increasing interest, also from the point of view of unitary representation theory (see e.g.~\cite{Neretin}). Notably, it was recently shown that the Neretin groups are not of type I \cite{CLBMB} (see also \cite{MR4449667}).

More generally, let $H$ be a non-compact, closed subgroup of $\mathrm{Aut}(T_d)$ acting transitively on the boundary $\partial T_d$ of $T_d$. Then $H$ satisfies the Kunze--Stein property \cite{MR0936361}. In \cite[Section 2]{MR3950641}, Lederle constructs the group $\mathcal{F}(H)$ of $H$-honest almost automorphisms, i.e.~almost automorphisms of $T_d$ that ``locally look like'' elements of $H$. She also shows that this group can be viewed as the topological full group of $H$ acting on $\partial T_d$. If it is moreover assumed that $H$ satisfies Tits' independence property, then the group $\mathcal{F}(H)$ carries a unique non-discrete, totally disconnected, locally compact group topology such that $H$ embeds into $\mathcal{F}(H)$ as an open subgroup \cite[Proposition 2.22]{MR3950641}. A special class of such groups $\mathcal{F}(H)$ is the class of coloured Neretin groups.

The unitary representation theory of $\mathrm{Aut}(T_d)$ goes back to \cite{MR0578650} (see also \cite{MR1152801}), and more generally, for subgroups $H$ of $\mathrm{Aut}(T_d)$ as above, it was described by Amann \cite{Amann}. It follows from Theorem \ref{thm:KSgp_inducedLpplus} and the above discussion that every unitary representation of $H$ can be induced to $\mathcal{F}(H)$ retaining the optimal $L^{p+}$-integrability of the representation.

We now make this more precise for the spherical complementary series of unitary representations under the additional assumption that the action of $H$ on $T_d$ is transitive. From the point of view of $L^{p+}$-representations, this part of the representation theory is most interesting. Indeed, for the groups $H$ under consideration, all non-spherical unitary irreducible representations are tempered (i.e.~they are $L^{2+}$-representations).

Let $H$ be a non-compact, closed subgroup of $\mathrm{Aut}(T_d)$ acting transitively on $T_d$ and on $\partial T_d$ and satisfying Tits' independence property. In what follows, we will use Amann's description \cite[Theorem 2]{Amann} of the unitary dual and refer to that text for details.

The unitary irreducible spherical representations are in one-to-one correspondence with the positive-definite spherical functions, and because the latter are diagonal matrix coefficients with respect to a cyclic vector, $L^p$-integrability (for some $p$) of the spherical function implies that the corresponding representation is an $L^p$-representation. As follows from (the proof of) \cite[Theorem 2]{Amann}, the positive-definite spherical functions associated with non-tempered representations are parametrised by the set $\mathcal{P}= \left[ 0,\frac{1}{2} \right)+ i\left\{ 0 ,\frac{\pi}{\log (d-1)} \right\}$ through a bijection
		\[
			\mathcal{P}\to \mathcal{SP}(K\backslash G/K), \; z \mapsto \varphi_{\gamma_o(z)}.
		\]
Let $\pi_{\gamma_0(z)}$ denote the (unique) spherical unitary irreducible representation associated with $\varphi_{\gamma_o(z)}$. For an explicit description of the spherical function $\varphi_{\gamma_o(z)}$, we refer to \cite{Amann}; we just recall its integrability properties. The following was shown in \cite[Lemma 4.11]{HdLS}: For $p\in (2,\infty)$ and $p' \in (1,2)$ with $\frac{1}{p}+\frac{1}{p'}=1$, the representation $\pi_{\gamma_o(z)}$ is an $L^{p+}(G)$-representation if and only if $\mathrm{Re} \, z \in \left[\frac{1}{p},\frac{1}{p'}\right]$. As a consequence of the above discussion and Theorem \ref{thm:KSgp_inducedLpplus}, we deduce the following.
\begin{thm}
Let $H$ be a non-compact, closed subgroup of $\mathrm{Aut}(T_d)$ acting transitively on $T_d$ and on $\partial T_d$ and satisfying Tits' independence property. Let $p \in (2,\infty)$ and $p' \in (1,2)$ be such that $\frac{1}{p} + \frac{1}{p'} = 1$. Then the representation of $\mathcal{F}(H)$ induced from $\pi_{\gamma_0(z)}$ is an $L^{p+}$-representation if and only if $\mathrm{Re} \, z \in \left[\frac{1}{p},\frac{1}{p'}\right]$. As a consequence, for every $p \in (2,\infty)$, the group $\mathcal{F}(H)$ has a unitary $L^{p+}$-representation that is not an $L^{q+}$-representation for any $q < p$. This gives an uncountable family of pairwise inequivalent non-tempered unitary representations of $\mathcal{F}(H)$.
\end{thm}

A version of the above theorem can also be formulated if the action of $H$ on $T_d$ is not transitive. However, due to an inaccuracy in the description of the unitary dual in \cite{Amann} in that case (as pointed out in e.g.~\cite[Remark 3.2.23]{Semal}), the parametrisation of the spherical complementary series of $H$ requires some more care.

\section{Remark on $L^{p+}$-group-$C^*$-algebras}
In the previous section, we constructed $L^{p+}$-representations of $\mathcal{F}(H)$ that are not $L^{q+}$-representations for $q < p$. It is interesting to point out that by $C^{\ast}$-algebraic methods, it follows that the $L^{p+}$-group-$C^*$-algebras $C^*_{L^{p+}}(\mathcal{F}(H))$ of $\mathcal{F}(H)$ must be pairwise canonically distinct for $p \in [2,\infty)$. More precisely, since $H$ is an open subgroup of $\mathcal{F}(H)$, it follows from Proposition \ref{prp:inducedlp}, \cite[End of Section 3]{MR3418075} and \cite[Theorem A]{HdLS} that for every $2 \leq q < p < \infty$, the canonical quotient map $C^*_{L^{p+}}(\mathcal{F}(H)) \to C^*_{L^{q+}}(\mathcal{F}(H))$ is not injective. Although this gives some information about the existence of $L^{p+}$-representations, this method does not give information about the optimal $L^{p+}$-integrability of specific representations (such as induced representations).

\section*{Acknowledgements}
\noindent We thank Pierre-Emmanuel Caprace and Sven Raum for useful comments and for pointing out \cite{MR4449667}. We also thank Pierre-Emmanuel Caprace for pointing out the inaccuracy in \cite{Amann}.

\end{document}